\documentclass[review]{elsarticle}

\usepackage{lineno,hyperref}
\modulolinenumbers[5]
\journal{Information and Computation}
\usepackage{graphicx}
\usepackage{proof,amsmath,amsthm,amssymb,mathrsfs} 

\newcommand{\BS}{\mathop{\backslash}}
\newcommand{\SL}{\mathop{/}}
\newcommand{\imp}{\multimap}
\newcommand{\Tensor}{\otimes}

\newcommand{\Th}{\mathrm{Th}}

\newcommand{\LL}{\mathbf{L}}
\newcommand{\LLs}{\mathbf{L}^{\boldsymbol{*}}}
\newcommand{\MALC}{\mathbf{MALC}}
\newcommand{\MALCs}{\mathbf{MALC}^{\boldsymbol{*}}}
\newcommand{\MALCd}{\mathbf{MALC}+\mathcal{D}}
\newcommand{\IAL}{\mathbf{IAL}}
\newcommand{\ILL}{\mathbf{ILL}}
\newcommand{\AMALC}{\mathbf{AMALC}^{\boldsymbol{*}}}
\newcommand{\One}{\mathbf{1}}
\newcommand{\LU}{\LL_\One}

\newcommand{\cBS}{\mathop{\mbox{\raisebox{5pt}{\rotatebox{-60}{${\sim}$}}}}}

\newcommand{\Mf}{\mathfrak{M}}

\newcommand{\Var}{\mathcal{V}}

\newcommand{\Lc}{\mathcal{L}}
\newcommand{\LLe}{\mathbf{L}^{+\varepsilon}}

\newtheorem{theorem}{Theorem}
\newtheorem{lemma}[theorem]{Lemma}
\newtheorem{corollary}[theorem]{Corollary}
\newtheorem{proposition}[theorem]{Proposition}
\theoremstyle{definition}
\newtheorem{definition}{Definition}

\begin{document}

\begin{frontmatter}

\title{Language Models for Some Extensions of the Lambek Calculus}
\author[ucl,hse]{Max Kanovich}
\author[mian,hse]{Stepan Kuznetsov}
\author[penn,hse]{Andre Scedrov}

\address[ucl]{University College London}
\address[mian]{Steklov Mathematical Institute of RAS}
\address[penn]{University of Pennsylvania}
\address[hse]{National Research University Higher School of Economics}

\begin{abstract}
 We investigate language interpretations of two extensions of the Lambek calculus: with additive conjunction and disjunction and with additive conjunction and the unit constant. For extensions with additive connectives, we show that conjunction and disjunction behave differently. Adding both of them leads to incompleteness due to the distributivity law. We show that with conjunction only no issues with distributivity arise. In contrast, there exists a corollary of the distributivity law in the language with disjunction only which is not derivable in the non-distributive system. Moreover, this difference keeps valid for systems with permutation and/or weakening structural rules, that is, intuitionistic linear and affine logics and affine multiplicative-additive Lambek calculus. For the extension of the Lambek with the unit constant, we present a calculus which reflects  natural algebraic properties of the empty word. We do not claim completeness for this calculus, but we prove undecidability for the whole range of systems extending this minimal calculus and sound w.r.t.\ language models. As a corollary, we show that in the language with the unit there exissts a sequent that is true if all variables are interpreted by regular language, but not true in language models in general.
\end{abstract}


\end{frontmatter}

\section{Introduction}

The Lambek calculus was introduced by Joachim Lambek~\cite{Lambek58} for mathematical modelling of natural language syntax.  This suggests the natural interpretation of the Lambek calculus as the algebraic logic of operations on formal languages. Such interpretations of the Lambek calculus are called {\em language models,} or {\em L-models} for short. 

The Lambek calculus, as originally formulated by Lambek, includes three operations: $\cdot$ (product), $\BS$ (left division), and $\SL$ (right division). A distinctive feature of the Lambek calculus is the so-called
{\em Lambek's non-emptiness restriction.} In terms of L-models, this means that the empty word is disallowed, and we consider, for a given alphabet $\Sigma$, subsets of $\Sigma^+$.
Lambek operations on languages are defined as follows:
\begin{align*}
 & A \cdot B = \{ uv \mid u \in A, v \in B \}, \\
 & A \BS B = \{ u \in \Sigma^+ \mid (\forall v \in A) \: 
 vu \in B \}, \\
 & B \SL A = \{ u \in \Sigma^+ \mid (\forall v \in A) \: 
 uv \in B \}.
\end{align*}

The division operations, $\BS$ and $\SL$, are indeed {\em residuals} of the product w.r.t.\ the subset relation:
$$
B \subseteq  A \BS C \iff
A \cdot B \subseteq C \iff 
A \subseteq C \SL B.
$$
These equivalences form the core of the Lambek calculus. 
Along with transitivity ($A \subseteq B \subseteq C \Rightarrow A \subseteq C$), reflexivity ($A \subseteq A$), and associativity ($A \cdot (B \cdot C) = (A \cdot B) \cdot C$), they form a complete axiomatization of all generally true atomic statements about Lambek operations on formal languages. This axiomatization is the Lambek calculus in its non-sequential form. 

The sequential formulation of the Lambek calculus~\cite{Lambek58} is as follows. Formulae are constructed from variables ($p,q,r,\ldots$) using three connectives: $\cdot$, $\BS$, $\SL$. (We use capital Latin letters both for languages and for Lambek formulae.) Sequents are expressions of the form $\Gamma \vdash C$, where the antecedent $\Gamma$ is a sequence of formulae and the succedent $C$ is one formula (intuitionistic style). 
The calculus $\LL$ includes axioms of the form $A \vdash A$ and the following  rules of inference:

$$
\infer[\BS L]{\Gamma, \Pi, A \BS B, \Delta \vdash C}
{\Pi \vdash A & \Gamma, B, \Delta \vdash C}
\qquad
\infer[\BS R;\mbox{ $\Pi$ is non-empty}]{\Pi \vdash A \BS B}
{A, \Pi \vdash B}
$$
$$
\infer[\SL L]{\Gamma, B \SL A, \Pi, \Delta \vdash C}
{\Pi \vdash A & \Gamma, B, \Delta \vdash C}
\qquad
\infer[\SL R;\mbox{ $\Pi$ is non-empty}]{\Pi \vdash B \SL A}
{\Pi, A \vdash B}
$$
$$
\infer[\cdot L]
{\Gamma, A \cdot B, \Delta \vdash C}{\Gamma, A, B, \Delta \vdash C}
\qquad
\infer[\cdot R]
{\Gamma, \Delta \vdash A \cdot B}
{\Gamma \vdash A & \Delta \vdash B}
$$
$$
\infer[Cut]
{\Gamma, \Pi, \Delta \vdash C}{\Pi \vdash A &
\Gamma, A, \Delta \vdash C}
$$
The cut rule is eliminable~\cite{Lambek58}.

An L-model, formally, is a mapping $w$ of Lambek formulae to subsets of $\Sigma^+$ (languages without the empty word), which commutes with Lambek operations: $w(A \cdot B) = w(A) \cdot w(B)$, $w(A \BS B) = w(A) \BS w(B)$, and $w(B \SL A) = w(B) \SL w(A)$. A sequent $A_1, \ldots, A_n \vdash B$ is true in this model, if $w(A_1) \cdot \ldots \cdot w(A_n) \subseteq w(B)$.

According to Lambek's non-emptiness restriction, all sequents in derivations are required to have non-empty antecedents. This constraint is motivated by linguistic applications: without it, Lambek categorial grammars generate ungrammatical sentences~\cite[\S\,2.5]{MootRetore}.


Abolishing Lambek's restriction---that is, removing constraints ``$\Pi$ is non-empty'' on $\BS R$ and $\SL R$---yields {\em the Lambek calculus allowing empty antecedents,} denoted by $\LLs$~\cite{Lambek61}.
Language models are easily adapted for the case of $\LLs$: now we consider languages, which are subsets of $\Sigma^*$ (that is, they are allowed to include the empty word $\varepsilon$). The definition of division operations is also modified: for models of $\LLs$, 
\begin{align*}
 & A \BS B = \{ u \in \Sigma^* \mid (\forall v \in A) \, vu \in B \},\\
 & B \SL A = \{ u \in \Sigma^* \mid (\forall v \in A) \, uv \in B \}.
\end{align*}
This modification can alter the values of $A \BS B$ and $B \SL A$ even if $A$ and $B$ do not contain the empty word. For example, $A \BS A$ now always includes $\varepsilon$, and therefore $(A \BS A) \BS B$ is always a subset of $B$. Hence, $\LLs$ is not a conservative extension of $\LL$: the sequent $(p \BS p) \BS q \vdash q$ has a non-empty antecedent, but is derivable only in $\LLs$, not in $\LL$. For these modified L-models, let us use the term {\em L$\varepsilon$-models.}

In an L$\varepsilon$-model $w$, a sequent of the form $A_1, \ldots, A_n \vdash B$ is true if $w(A_1) \cdot \ldots \cdot w(A_n) \subseteq w(B)$, and a sequent of the form $\vdash B$, with an empty antecedent, is true if $\varepsilon \in w(B)$.

Completeness theorems for $\LL$ and $\LLs$ w.r.t.\ corresponding versions of L-models were  proved by Pentus~\cite{PentusAPAL,PentusFmonov}. Pentus' proofs are 
highly non-trivial. If one considers the fragment without $\cdot$ (the {\em product-free} fragment), however, proving
L-completeness becomes much easier. This was done by Buszkowski~\cite{Buszko1982}; Buszkowski's proof applies both to $\LL$ and $\LLs$, w.r.t.\ L-models and L$\varepsilon$-models, respectively.

Besides product and two divisions, natural operations on formal languages include set-theoretic intersection and union. These operations correspond to so-called {\em additive} conjunction and disjunction. Additive operations are usually axiomatized by the following inference rules (cf.~\cite{KanazawaJoLLI}):
%
%
$$
\infer[\vee L]
{\Gamma, A \vee B, \Delta \vdash C}
{\Gamma, A, \Delta \vdash C & \Gamma, B, \Delta \vdash C}
\qquad
\infer[\vee R_l]
{\Pi \vdash A \vee B}{\Pi \vdash A}
\qquad
\infer[\vee R_r]
{\Pi \vdash A \vee B}{\Pi \vdash B}
$$
$$
\infer[\wedge L_l]
{\Gamma, A \wedge B, \Delta \vdash C}
{\Gamma, A, \Delta \vdash C}
\qquad
\infer[\wedge L_r]
{\Gamma, A \wedge B, \Delta \vdash C}
{\Gamma, B, \Delta \vdash C}
\qquad
\infer[\wedge R]
{\Pi \vdash A \wedge B}
{\Pi \vdash A & \Pi \vdash B}
$$
The Lambek calculus $\LL$ extended with these rules is denoted by $\MALC$ ({\em mul\-ti\-pli\-ca\-ti\-ve-additive Lambek calculus}); $\MALCs$ is the variant of $\MALC$ without Lambek's restriction (that is, allowing empty antecedents).
L-completeness, however, fails for $\MALC$ in general. Further, in Section~\ref{S:distr}, we discuss this issue in detail.

Following Abrusci~\cite{Abrusci}, we put the Lambek calculus into a broader context of linear logic. 
Namely, $\MALCs$ can be viewed as a fragment of intuitionistic {\em non-commutative} linear logic. (This fragment includes multiplicative and additive operations, but lacks the exponential and constants.) We also consider commutative systems: intuitionistic linear logic $\ILL$ and intuitionistic affine logic $\IAL$.

Calculi $\ILL$ and $\IAL$ can be obtained from $\MALCs$ by adding structural rules: permutation for $\ILL$ and permutation and weakening for $\IAL$. 
In the language of $\MALC$, the rules of permutation and weakening are formulated as follows:
$$
\infer[P]{\Gamma, A, B, \Delta \vdash C}{\Gamma, B, A, \Delta \vdash C}
\qquad
\infer[W]{\Gamma, A, \Delta \vdash C}{\Gamma, \Delta \vdash C}
$$
Adding only weakening yields non-commutative intuitionistic affine logic, or affine (monotone) multiplicative-additive Lambek calculus. We denote this system by $\AMALC$ (in the presence of extra structural rules, we do not impose Lambek's restriction).

We shall also use alternative calculi for the commutative systems $\ILL$ and $\IAL$, in which structural rules are hidden in axioms and in the format of sequents. 
First, we change the language of formulae, introducing one connective $A \imp B$ instead of $A \BS B$ and $B \SL A$ (these are equivalent in $\ILL$ and $\IAL$). We also write $A \otimes B$ instead of $A \cdot B$, following Girard's~\cite{Girard} linear logic notations.

Sequents are now going to be expressions of the form $\Gamma \vdash C$, where $\Gamma$ is a {\em multiset} of formulae. Further $\Gamma, A$ means $\Gamma \uplus \{ A \}$, and $\Gamma, \Pi$ means $\Gamma \uplus \Pi$, where $\uplus$ is multiset union.

Axioms are of the form $p \vdash p$, for each variable $p$, in the case of $\ILL$, and of the form $\Gamma, p \vdash p$ for $\IAL$. Inference rules for both systems are as follows:
$$
 \infer[\imp L]{\Gamma, \Pi, A \imp B \vdash C}
 {\Pi \vdash A & \Gamma, B \vdash C}
 \qquad 
 \infer[\imp R]{\Pi \vdash A \imp B}{\Pi, A \vdash B}
 $$
 $$
 \infer[\Tensor L]{\Gamma, A \Tensor B \vdash C}{\Gamma, A, B \vdash C}
 \qquad
 \infer[\Tensor R]{\Gamma, \Delta \vdash A \Tensor B}
 {\Gamma \vdash A & \Delta \vdash B}
 $$
 $$
 \infer[\vee L]{\Gamma, A \vee B \vdash C}{\Gamma, A \vdash C &
 \Gamma, B \vdash C}
 \qquad
 \infer[\vee R_l]{\Pi \vdash A \vee B}{\Pi \vdash A}
 \qquad
 \infer[\vee R_r]{\Pi \vdash A \vee B}{\Pi \vdash B}
 $$
 $$
 \infer[\wedge L_l]{\Gamma, A \vee B \vdash C}{\Gamma, A \vdash C}
 \qquad
 \infer[\wedge L_r]{\Gamma, A \vee B \vdash C}{\Gamma, B \vdash C}
 \qquad 
 \infer[\wedge R]{\Pi \vdash A \wedge B}{\Pi \wedge A & \Pi \wedge B}
 $$
 
 For $\IAL$, the weakening rule is not officially included in the system, but is admissible:
 $$
 \infer[W]{\Gamma, A \vdash C}{\Gamma \vdash C}
 $$
 (it is hidden in axioms). 
 
 The cut rule of the following form is admissible both in $\ILL$ and $\IAL$:
 $$
 \infer[Cut]{\Gamma, \Pi \vdash C}{\Pi \vdash A & \Gamma, A \vdash C}
 $$
 This is shown by a standard inductive argument.
 
Finally, let us introduce the {\em multiplicative unit constant,} $\One$. The unit constant is added to systems without Lambek's restriction extending $\LLs$ ({\em i.e.,} $\LLs$ itself, $\MALCs$, $\AMALC$, $\ILL$, $\IAL$). 
The {\em Lambek calculus with the unit,} $\LU$~\cite{Lambek69}, is obtained from $\LLs$ by adding one axiom, $\vdash \One$ (its antecedent is empty), and one inference rule,
$$
\infer[\One L]{\Gamma, \One, \Delta \vdash C}
{\Gamma, \Delta \vdash C}
$$
L-completeness, however, does not hold for $\LU$. Indeed, 
since $\One$ should be the unit w.r.t.\ $\cdot$, that is $A \cdot \One = A = \One \cdot A$ for any $A$, in L$\varepsilon$-models it should be interpreted as $\{ \varepsilon \}$. The following sequent is a counter-example for L-completeness: 
$\One \SL p, \One \SL p \vdash \One \SL p$. This sequent is true in all models for any interpretation of $p$, but is not derivable in $\LU$.

Throughout this paper, we shall frequently consider fragments of the calculi defined above in languages with restricted sets of connectives. Such a fragment will be denoted by the name of the calculus, followed by the list of connectives in parentheses: e.g., $\MALC(\BS,\SL,\wedge)$.


\section{Distributivity Law in Fragments with One Additive}\label{S:distr}

It is well known, that $\MALC$ is incomplete w.r.t.\ L-models. The reason is the distributivity
principle, 
$$
(A \vee C) \wedge (B \vee C) \vdash (A \wedge B) \vee C.
\eqno{(\mathcal{D})}
$$
On one hand, this principle is obviously true in all L-models. On the other hand, as noticed by Ono and Komori~\cite{OnoKomori}, one needs the structural rules of contraction and weakening to derive it. In particular, the distributivity principle is not derivable in $\MALC$,
$\MALCs$, $\AMALC$, $\ILL$, and $\ILL$.

The distributivity principle, as formulated above, includes {\em both} additive connectives, $\wedge$ and $\vee$. We investigate fragments of $\MALC$ with only one additive, $\wedge$ {\em or} $\vee$. The result of our study is that with respect to distributivity $\wedge$ and $\vee$ behave in opposite ways.

Let $\MALCd$ denote $\MALC$ with the distributivity principle added as an extra axiom scheme. 
In the presence of this axiom scheme, we have to keep cut as an official rule of the system (it is now not eliminable). 
A hypersequential system for $\MALCd$ was developed by Kozak~\cite{Kozak}.

Let us restrict ourselves to the product-free language (with product, proving L-completeness is hard even without extra connections~\cite{PentusAPAL,PentusFmonov}). We also consider calculi without the unit constant: issues connected with $\One$ are discussed in Section~\ref{S:unit}. Thus, we consider two fragments of the multiplicative-additive Lambek calculus: $\MALC(\BS,\SL,\wedge)$ and $\MALC(\BS,\SL,\vee)$, and the corresponding fragments of bigger system up to $\IAL$. (For commutative calculi, we have only one implication, that is, consider fragments in the language of $\imp, \wedge$ and $\imp, \vee$.)

As shown by Buszkowski~\cite{Buszko1982}, $\MALC(\BS,\SL,\wedge)$ is complete w.r.t.\ L-models. This yields the following corollary: $\MALC(\BS,\SL,\wedge)$ is a conservative fragment of {\em both} $\MALC$ and $\MALCd$. Indeed, any sequent provable in $\MALCd$ is true in all L-models; if it is in the language of $\BS,\SL,\wedge$, it is derivable in $\MALC(\BS,\SL,\wedge)$ by L-completeness.
In other words, the distributivity principle has no non-trivial corollaries in the language of $\BS,\SL,\wedge$.

The situation with $\MALC(\BS,\SL,\vee)$ is opposite. Namely, we present a corollary of the distributivity principle in the language of $\BS,\SL,\vee$, which is not provable in $\MALC(\BS,\SL,\vee)$. Thus, $\MALC(\BS,\SL,\vee)$ is not a conservative fragment of $\MALCd$, and is therefore incomplete w.r.t.\ L-models.
Moreover, we show that this effect is of a more general nature. Namely, the same holds for the corresponding fragments of $\MALCs$, $\AMALC$, $\ILL$, and $\IAL$: distributivity has no new corollaries in the language with $\wedge$, but has such in the language with $\vee$.

\subsection{Completeness with Additive Conjunction Only}

For the first series of results, concerning $\wedge$, we give a semantic proof. For each system, we consider a specific version of L-semantics. For $\MALC(\BS,\SL,\wedge)$ and $\MALCs(\BS,\SL,\wedge)$, these are L-models and L$\varepsilon$-models respectively. For other systems, let us first give some definitions and prove correctness statements for them.

\begin{definition}
 A language $A$ is called {\em monotone,} if for any word $u_1 u_2 \in A$ and an arbitrary word $w$ the word $u_1 w u_2$ also belongs to $A$.
\end{definition}

\begin{proposition}
If $A$ and $B$ are both monotone, then so are $A \BS B$, $B \SL A$, and $A \wedge B$.
\end{proposition}


\begin{proof}
 Let $u  = u_1 u_2 \in A \BS B$. Then for any $v \in A$ we have $v u_1 u_2 \in B$. Now take $u' = u_1 w u_2$ for an arbitrary $w$. By monotonicity of $B$, the word $vu' = v u_1 w u_2$ is also in $B$. Since this holds for any $v \in A$, we get $u' \in A \BS B$.
 The reasoning for $B \SL A$ is symmetric. The case of $A \wedge B$ is trivial.
\end{proof}

\begin{definition}
%
 A language $A$ is called {\em commutative,} if for any word $u = a_1 \ldots a_n$ belonging to $A$ and an arbitrary transposition $\sigma \in \mathbf{S}_n$ on $\{ 1, \ldots, n \}$ the word $a_{\sigma(1)} \ldots a_{\sigma(n)}$ also belongs to $A$.
\end{definition}


Commutative languages are in one-to-one correspondence with multisets of letters from $\Sigma$. 
Thus, we can define the operation of {\em multiset union,} $A \uplus B$, for two commutative languages $A$ and $B$, which can be expressed as follows:
$$
A \uplus B = \{ a_{\sigma(1)} \ldots a_{\sigma(n)} \mid
\sigma \in \mathbf{S}_n \mbox{ and } 
a_1 \ldots a_n \in A \cdot B \}.
$$

If $A$ is a commutative language, then $vu \in A$ if and only if $uv \in A$. Therefore, for commutative $A$ and $B$, we have $A \BS B = B \SL A$; we denote this language by $A \imp B$.

\begin{proposition}
 If $A$ and $B$ are commutative, then so is $A \imp B$ and $A \wedge B$.
\end{proposition}

\begin{proof}
 Commutativity of $A \wedge B$ is obvious. For $A \imp B$, take any $u = a_1 \ldots a_n \in A \imp B = B \SL A$ and let $u' = u_{\sigma(1)} \ldots u_{\sigma(n)}$. Now for any $v = a_{n+1} \ldots a_{m} \in A$. By definiton of $B \SL A$, we have
 $uv \in B$. Now by commutativity of $B$, the word $u'v$ also belongs to $B$. Indeed, it is obtained from $uv$ by the following transposition:
 $$
 \tilde{\sigma} =
 \left(
 \begin{matrix}
  1 & 2 & \ldots & n & n+1 & \ldots & m \\
  \sigma(1) & \sigma(2) & \ldots & \sigma(n) & n+1 & \ldots & m
 \end{matrix}\right).
 $$
 Since $v \in A$ was taken arbitrarily, we conclude that $u' \in B \SL A = A \imp B$.
\end{proof}

Having the class of monotone languages and the class of commutative languages closed under our operations ($\BS$, $\SL$, $\wedge$), we can define the classes of restricted L$\varepsilon$-models for all our systems.

\begin{definition}
 An L$\varepsilon$-model is monotone, if all languages in it are monotone. Truth of sequents is defined as in ordinary L$\varepsilon$-models.
\end{definition}

\begin{definition}
 A commutative L$\varepsilon$-model is an L$\varepsilon$-model, where all languages are commutative.  
\end{definition}

In commutative models $\uplus$ actually plays the role of product (while we do not have product as a connective, we still have the sequential comma, which is a hidden product), due to the following fact.

\begin{proposition}\label{P:commtruth}
 In a commutative L$\varepsilon$-model $w$, a sequent $A_1, \ldots, A_n \vdash B$ is true if and only if $w(A_1) \uplus \ldots \uplus w(A_n) \subseteq w(B)$.
\end{proposition}

\begin{proof}
 The ``if'' part is due to the fact that $w(A_1) \cdot \ldots \cdot w(A_n) \subseteq w(A_1) \uplus \ldots \uplus w(A_n)$. The ``only if'' part holds since $w(B)$ is closed under transpositions.
\end{proof}

Now we prove an extension of Buszkowski's completeness result 

\begin{theorem}\label{Th:compl}
 Each of $\MALC(\BS,\SL,\wedge)$, $\MALCs(\BS,\SL,\wedge)$, $\AMALC(\BS,\SL,\wedge)$, $\ILL(\imp,\wedge)$, $\IAL(\imp,\wedge)$ is sound and complete w.r.t.\ the corresponding class of models, according to the following table:
 {\rm
 \begin{center}
 \begin{tabular}{c|c} 
  {\em Calculus} & {\em Models} \\ \hline
  $\MALC(\BS,\SL,\wedge)$ & L-models \\
  $\MALCs(\BS,\SL,\wedge)$ & L$\varepsilon$-models \\
  $\AMALC(\BS,\SL,\wedge)$ & monotone L$\varepsilon$-models \\
  $\ILL(\imp,\wedge)$ & commutative L$\varepsilon$-models \\
  $\IAL(\imp,\wedge)$ & L$\varepsilon$-models, which are both monotone and commutative
 \end{tabular}
\end{center}}
\end{theorem}

\begin{proof}
The cases of $\MALC(\BS,\SL,\wedge)$ and $\MALCs(\BS,\SL,\wedge)$ are due to Busz\-kow\-ski~\cite{Buszko1982}. Let us consider the remaining three systems.

 The soundness part is easy: our conditions on models were specifically designed to reflect structural rules. In a monotone model, if $w(A_1) \cdot \ldots \cdot w(A_k) \cdot w(A_{k+1}) \cdot \ldots \cdot w(A_n) \subseteq w(B)$, then also
 $w(A_1) \cdot \ldots \cdot w(A_k) \cdot w(A) \cdot w(A_{k+1}) \cdot \ldots \cdot w(A_n) \subseteq w(B)$, thus the weakening rule is valid. 
  If we have a commutative L$\varepsilon$-model, then the permutation rule is valid. This is obvious from  Proposition~\ref{P:commtruth}: unlike $\cdot$, $\uplus$ is just commutative. All other rules and axioms are valid in arbitrary L$\varepsilon$-models.
  
  Completeness is proved by Buszkowski's canonical model argument. We do it uniformly for all systems. In the canonical model, the alphabet $\Sigma$ is the set of all formulae of the given calculus, and for any formula $A$ let 
  $$
  w_0(A) = \{ \Gamma \mid \Gamma \vdash A \mbox{ is derivable in the given system} \}.
  $$
  
  First we show that $w_0$ is indeed an L${\varepsilon}$-model:
  \begin{align*}
   & w_0(A \BS B) = w_0(A) \BS w_0(B); \\
   & w_0(B \SL A) = w_0(B) \SL w_0(A); \\
   & w_0(A \wedge B) = w_0(A) \wedge w_0(B).
  \end{align*}
  This is performed exactly as in Buszkowski's proof. Indeed, if $\Gamma \in w_0(A \BS B)$, then for an arbitrary $\Delta \in w_0(A)$ we have $\Gamma \vdash A \BS B$ and $\Delta \vdash A$. Applying cut with $A, A \BS B \vdash B$, we get $\Delta, \Gamma \vdash A$ derivable in our system. Thus, $\Delta\Gamma \in w_0(B)$, therefore $\Gamma \in w_0(A) \BS w_0(B)$. Notice that cut is available in all systems we consider. Dually, if $\Gamma \in w_0(A) \BS w_0(B)$, then, since $A \in w_0(A)$ by the axiom, $A\Gamma \in w_0(B)$. This means derivability $A, \Gamma \vdash B$, thus $\Gamma \vdash A \BS B$. Hence, $\Gamma \in w_0(A \BS B)$.
  
  The $\SL$ case is symmetric. For $\wedge$, we use the equivalence $\Gamma \vdash A \wedge B$ if and only if $\Gamma \vdash A$ and $\Gamma \vdash B$. Here the ``if'' part is an application of $\wedge R$, and the ``only if'' part is by cut with $A \wedge B \vdash A$ and $A \wedge B \vdash B$.

  Next,  is easy to see that the canonical model $w_0$ belongs to the corresponding class of models: monotone for $\AMALC(\BS,\SL,\wedge)$, commutative for $\ILL(\BS,\SL,\wedge)$, commutative and monotone for $\IAL(\BS,\SL,\wedge)$.
  
  Finally, suppose a sequent $\Pi \vdash B$ is not derivable. Consider two cases. If $\Pi = A_1, \ldots, A_n$ is non-empty, then, since each $A_i$ belongs to $w(A_i)$, we have $\Gamma \in w(A_1) \cdot \ldots \cdot w(A_n)$. On the other hand, $\Gamma \notin w(B)$. This falsifies $\Pi \vdash B$ under interpretation $w_0$. If $\Pi$ is empty, then we have $\varepsilon \notin w(B)$, which again falsifies $\Pi \vdash B$. This finishes the completeness proof.
\end{proof}

It is easy to see that soundness actually extends to the language with $\vee$ (interpreted as set-theoretic union). Unions of monotone languages are also monotone, the same for commutative languages. The situation with product is a bit more complicated for commutative systems, since $A \cdot B$ is usually not commutative, even for commutative $A$ and $B$. Thus, we  have to alter the definition of language models in the commutative case, requiring $w(A \cdot B) = w(A) \uplus w(B)$ instead of $w(A \cdot B) = w(A) \cdot w(B)$. Under this modification, soundness holds for product also. Finally, notice that in all models we consider $\vee$ and $\wedge$ are interpreted set-theoretically, thus, obey the distributivity law. These considerations yield the following soundness result:

\begin{proposition}\label{P:sound}
 Each of $\MALC + \mathcal{D}$, $\MALCs + \mathcal{D}$, $\AMALC + \mathcal{D}$, $\ILL + \mathcal{D}$, $\IAL + \mathcal{D}$ is sound w.r.t.\ the corresponding class of models, according to the table in Theorem~\ref{Th:compl}; for $\ILL$ and $\IAL$ in the models we use $\uplus$ to interpret $\cdot$.
\end{proposition}

Now we are ready to state and prove our conservativity result.

\begin{theorem}\label{Th:conservdisj}
 The systems in the restricted language without $\vee$, $\MALC(\BS,\SL,\wedge)$, $\MALCs(\BS,\SL,\wedge)$, $\AMALC(\BS,\SL,\wedge)$, $\ILL(\imp,\wedge)$, and 
 \mbox{$\IAL(\imp,\wedge)$} are conservative fragments of
 $\MALC + \mathcal{D}$, $\MALCs + \mathcal{D}$, $\AMALC + \mathcal{D}$, $\ILL + \mathcal{D}$, and $\IAL + \mathcal{D}$ respectively.
\end{theorem}

\begin{proof}
Let $\Pi \vdash B$ be a sequent in the language of $\BS, \SL, \wedge$ (in the commutative case, $\imp, \wedge$).
 Suppose it is derivable in one of the distributive systems, 
 $\MALC + \mathcal{D}$, \ldots, $\IAL + \mathcal{D}$.  Then by Proposition~\ref{P:sound} it is true in all models of the corresponding class. By Theorem~\ref{Th:compl} it is derivable in, respectively, $\MALC(\BS,\SL,\wedge)$, \ldots, $\IAL(\imp,\wedge)$.
\end{proof}

\subsection{Incompleteness with Additive Disjunction Only}

If we take $\vee$ instead of $\wedge$, however, no analog of the conservativity result like Theorem~\ref{Th:conservdisj} is possible, due to the following counter-example.


 
 \begin{theorem}\label{Th:disj}
  The sequent 
  \begin{multline*}
  ((x \SL y) \vee w) \SL ((x \SL y) \vee (x \SL z) \vee w),
  (x \SL y) \vee w, \\ 
  ((x \SL y) \vee w) \BS ((x \SL z) \vee w) \vdash (x \SL (y \vee z)) \vee w
  \end{multline*}
  is derivable in $\MALCd$ 
  but this sequent 
  is not derivable in $\IAL$.
 \end{theorem}
 
 This sequent is in the language of $\BS,\SL,\vee$. The theorem states that it is derivable in $\MALC + \mathcal{D}$, and therefore in all its extensions up to $\IAL + \mathcal{D}$, but not in the corresponding ($\BS,\SL,\vee$) fragments without the distributivity law added. Thus, this is a non-trivial corollary of $\mathcal{D}$ in the language without $\wedge$.
 In particular, Theorem~\ref{Th:disj} implies that $\MALC(\BS,\SL,\vee)$ is incomplete w.r.t.\ L-models, as well as $\MALCs(\BS,\SL,\vee)$, $\AMALC(\BS,\SL,\vee)$, $\ILL(\imp,\vee)$, $\IAL(\imp,\vee)$ are incomplete w.r.t.\ the corresponding modifications of L-models (compare with Theorem~\ref{Th:compl}).
 
 Before proving Theorem~\ref{Th:disj}, let us make some remarks. First, let us notice that the sequent in this theorem is slightly different from the one in our WoLLIC 2019 paper~\cite{KanKuzSceWoLLICLMod}, where one variable is used for $x$ and $w$. The reason is that the old example happens to be derivable in $\IAL$ (but still not in $\ILL$ and weaker systems). 
 
 Second, the hard part of Theorem~\ref{Th:disj} is, of course, the second one (non-derivability). Fortunately, the derivability problem in $\MALC$ is algorithmically decidable (belongs to PSPACE), thus, it is possible to establish non-derivability by exhaustive proof search. This proof search was first performed, as a pre-verification of the result, automatically using an affine modification of {\tt llprover} by Tamura~\cite{Tamura}. (For the WoLLIC 2019 paper, we used a $\MALC$ prover by Jipsen~\cite{Jipsenprover}, based on the algorithm by Okada and Terui~\cite{OkadaTerui}.) In order to make this article self-contained and independent from proof-search software, here we present a complete manual proof search. 
 
 One of the WoLLIC 2019 reviewers suggested a shorter method of proving non-derivability of the given sequent in $\MALC$, via an algebraic counter-model. This counter-model is a commutative residuated lattice on the set $R = \{ 0, a, b, c, 1 \}$. The order is defined as follows: $0 \prec a,b,c \prec 1$; $a,b,c$ are incomparable. Product and residual are defined as follows:
 $$\begin{array}{c|ccccc}
\cdot&0&a&b&c&1\\\hline
0&0&0&0&0&0\\
a&0&a&b&c&1\\
b&0&b&a&c&1\\
c&0&c&c&0&c\\
1&0&1&1&c&1\\
\end{array}
\qquad
\begin{array}{c|ccccc}
\imp&0&a&b&c&1\\\hline
0&1&1&1&1&1\\
a&0&a&b&c&1\\
b&0&b&a&c&1\\
c&c&c&c&1&1\\
1&0&0&0&c&1\\
\end{array}$$
(In the commutative situation, we have only one residual, which we denote by $\imp$.) Variables are interpreted as follows: $y$ as $b$, $z$ as $c$, $x$ and $w$ both as $a$.
This algebraic model falsifies the sequent in Theorem~\ref{Th:disj}. However, is insufficient for our new purposes. The reason is that in this model $a \cdot b = b \not\preceq a$, while in the presence of weakening $A \cdot B \vdash A$ should be true. Thus, in order to establish non-derivability of our sequent not only in $\MALC$, but also in $\IAL$, we use the good old syntactic method. 
 
 
 

 \begin{proof}[Proof of Theorem~\ref{Th:disj}]
 The first statement is proved using the joining (diamond) construction, the idea of which goes back to Lambek~\cite{Lambek58} and Pentus~\cite{PentusJoLLI}. Indeed, let $A = (x \SL y) \vee w$ and $B = (x \SL z) \vee w$. Then $A \vee B$ is equivalent to $(x \SL y) \vee (x \SL z) \vee w$. One can easily check derivability of $A \SL (A \vee B), A, A \BS B \vdash A$ and
 $A \SL (A \vee B), A, A \BS B \vdash B$ in $\MALC$. Notice that the antecedent of this sequent is exactly the one in the sequent of our theorem. Next, we derive $A \SL (A \vee B), A, A \BS B \vdash A \wedge B$, and further by distributivity
 $A \wedge B \equiv ((x \SL y) \wedge (x \SL z)) \vee w \equiv
 (x \SL (y \vee z)) \wedge w$.
 
 The second statement is proved by an exhaustive proof search for the sequent
 \begin{multline*}
 ((y \imp x) \vee (z \imp x) \vee w) \imp ((y \imp x) \vee w),
 (y \imp x) \vee w, \\
 ((y \imp x) \vee w) \imp ((z \imp x) \vee w)
 \vdash ((y \vee z) \imp x) \vee w
 \end{multline*}
 (the  translation of our sequent into the commutative language) in $\IAL$.
 
 In order to facilitate proof search, we take into account the following considerations.
 
 First, the rules $\vee L$ and $\imp R$ are invertible. Thus, we can suppose they are applied immediately. Moreover, $\vee L$ has two premises, and when disproving derivability we have the right to choose one and establish non-derivability there.
 
 Second, we can suppose that in our (hypothetic) derivation instances of $\vee L_r$ of the form $\dfrac{\Gamma \vdash w}{\Gamma \vdash F \vee w}$ are directly preceded by axioms. Indeed, such instances are interchangeable upwards with $\imp L$ and $\vee L$, and $\imp R$ cannot appear before this $\vee L_r$, since $w$ is a variable. Other rules are impossible by the polarized subformula property.
 
 Third, we establish non-derivability of several sequents, which will appear frequently in our proof search:
 \begin{equation}\label{yxw}
 \not\vdash (y \imp x) \vee w
 \end{equation}
 \begin{equation}\label{zyxw}
  z \not\vdash (y \imp x) \vee w
 \end{equation}
 \begin{equation}\label{yyxw}
 y \not\vdash (y \imp x) \vee w
 \end{equation}
 \begin{equation}\label{zyyxw}
  z,y \not\vdash (y \imp x) \vee w
 \end{equation}
 \begin{equation}\label{zzyxw}
  z,z \not\vdash (y \imp x) \vee w
 \end{equation}
 \begin{equation}\label{zyxzxw}
 z \not\vdash (y \imp x) \vee (z \imp x) \vee w
 \end{equation}
 \begin{equation}\label{yxzxw}
 \not\vdash (y \imp x) \vee (z \imp x) \vee w
 \end{equation}
 \begin{equation}\label{zyyxzxw}
 z, y \not\vdash (y \imp x) \vee (z \imp x) \vee w
 \end{equation}

 Now we are ready to start proof search. First we invert $\vee L$ introducing $(y \imp x) \vee w$ and choose $y \imp x$:
 \begin{multline*}
 ((y \imp x) \vee (z \imp x) \vee w) \stackrel{2}{\imp}
 ((y \imp x) \vee w),
 y \stackrel{3}{\imp} x, \\
 ((y \imp x) \vee w) \stackrel{4}{\imp} ((z \imp x) \vee w)
 \vdash ((y \vee z) \imp x) \stackrel{1}{\vee} w
 \end{multline*}
 
 Now we have a choice of 4 principal connectives (denoted by numbers in the sequent) to be decomposed first.
 
 {\bf Case 1.} In this case, we use $\vee R_l$, thanks to our consideration that $\vee R_r$ with $w$ should be applied immediately after an axiom.
 \begin{multline*}
 ((y \imp x) \vee (z \imp x) \vee w) \stackrel{2}{\imp}
 ((y \imp x) \vee w),
 y \stackrel{3}{\imp} x, \\
 ((y \imp x) \vee w) \stackrel{4}{\imp} ((z \imp x) \vee w)
 \vdash (y \vee z) \imp x
 \end{multline*}
Invert $\imp R$ and $\vee L$, choosing $z$ out of $y \vee z$:
\begin{multline*}
((y \imp x) \vee (z \imp x) \vee w) \stackrel{2}{\imp}
 ((y \imp x) \vee w),
 y \stackrel{3}{\imp} x,\\
 ((y \imp x) \vee w) \stackrel{4}{\imp} ((z \imp x) \vee w), z \vdash x
\end{multline*}

Now we can decompose (by $\imp L$) one of the implications 2--4, and for each we have a choice of $8 = 2^3$ ways of splitting the rest of the antecedent into $\Pi$ and $\Gamma$. Making use of the weakening rule, however, we can reduce the number of cases.

{\bf Subcase 1--2.} If $\Pi$ includes $y \stackrel{3}{\imp x}$, then the right premise is $\Gamma, (y \imp x) \vee w \vdash x$, where $\Gamma$ is a subset of $z, ((y \imp x) \vee w) \imp ((z \imp x) \vee w)$.
Notice that if $\Gamma' \subseteq \Gamma$ and the sequent is not derivable with $\Gamma$, it is also not derivable with $\Gamma'$ (otherwise we could derive it with $\Gamma$ using the weakening rule).
However, the sequent is not derivable even with the maximal $\Gamma$:
$$
z, ((y \imp x) \vee w) \imp ((z \imp x) \vee w), (y \imp x) \vee w
\not\vdash x.
$$
Indeed, invert $\vee L$ and choose $w$:
$$
z, w, ((y \imp x) \vee w) \imp ((z \imp x) \vee w) \not \vdash x.
$$
Here one should use $\imp L$, 
 but then in its right premise we can again invert $\vee L$ choosing $w$, which yields one of:
 $$
 w \vdash x \qquad z,w \vdash x \qquad w,w \vdash x \qquad
 z,w,w \vdash x.
 $$
 None of these is derivable.
 
 If $\Pi$ does not include $y \stackrel{3}{\imp} x$, 
 then $\Pi$ is a subset of
 $((y \imp x) \vee w) \imp ((z \imp x) \vee w), z$, and we again take the maximal $\Pi$ in the left premise:
 \begin{equation}\label{three}
 ((y \imp x) \vee w) \imp ((z \imp x) \vee w), z \vdash (y \imp x) \vee (z \imp x) \vee w
 \end{equation}
 Decomposing $\imp$ yields either $\vdash (y \imp x) \vee w$ or
 $z \vdash (y \imp x) \vee w$, both not derivable by~(\ref{yxw}) and~(\ref{zyxw}). Thus, we have to decompose $\vee$ on the right. 
 
 Taking $y \imp x$ (and inverting $\imp R$) yields
 $$
 ((y \imp x) \vee w) \imp ((z \imp x) \vee w), z, y \vdash x.
 $$
 Now we again have to use $\imp L$. The new cases are $y \vdash (y \imp x) \vee w$ and $z, y \vdash (y \imp x) \vee w$, both not derivable~(\ref{yyxw})(\ref{zyyxw}).
 
 Taking $z \imp x$ and inverting $\imp R$ gives
 $$
 ((y \imp x) \vee w) \imp ((z \imp x) \vee w), z, z \vdash x.
 $$
 Decomposing $\imp$ fails due to~(\ref{yxw})(\ref{zyxw})(\ref{zzyxw}).
 
 {\bf Subcase 1--3.} Apply $\imp L$ for $\stackrel{3}{\imp}$ and
 consider its left premise with the maximal possible $\Pi$:
 \begin{equation}\label{vorpal}
 ((y \imp x) \vee (z \imp x) \vee w) \stackrel{2}{\imp}
 ((y \imp x) \vee w), ((y \imp x) \vee w) \stackrel{4}{\imp}
 ((z \imp x) \vee w), z \vdash y.
 \end{equation}
 
 {\em Subsubcase 1--3--2.} Decompose $\stackrel{2}{\imp}$. If the big formula with $\stackrel{4}{\imp}$ goes to the new $\Gamma$,
 then the new $\Pi$ is either $z$ or empty. However, neither $z \vdash (y \imp x) \vee (z \imp x) \vee w$ nor 
 $\vdash (y \imp x) \vee (z \imp x) \vee w$ is derivable~(\ref{zyxzxw})(\ref{yxzxw}).
 If the formula with $\stackrel{4}{\imp}$ goes to the new $\Pi$, then the new $\Gamma$ is either $z$ or empty. This gives, at maximum, $z, (y \imp x) \vee w \vdash y$, which is falsified by choosing $w$ in the inverted $\vee L$: $z, w \not\vdash y$.
 
 {\em Subsubcase 1--3--4.} Decompose $\stackrel{4}{\imp}$. Again, if the big formula (now with $\stackrel{2}{\imp}$) goes to the new $\Gamma$, we falsify the left premise by~(\ref{yxw}) or~(\ref{zyxw}). Otherwise, the right premise is, at maximum, $z, (z \imp x) \vee w \vdash y$, which is again falsified by choosing $w$.
 
 {\bf Subcase 1--4.} 
 If $\Pi$ includes $y \stackrel{3}{\imp} x$, then the right premise is, at maximum,
 $$
 ((y \imp x) \vee (z \imp x) \vee w) \stackrel{2}{\imp}
 ((y \imp x) \vee w), z, (z \imp x) \vee w \vdash x
 $$
 Invert $\vee L$ and choose $w$:
 $$
 ((y \imp x) \vee (z \imp x) \vee w) \stackrel{2}{\imp}
 ((y \imp x) \vee w), z, w \vdash x
 $$
 Now we have to use $\stackrel{2}{\imp} L$. Its right premise is, at maximum, $z,w,(y \imp x) \vee w \vdash x$. Choosing $w$ falsifies it.
 
 If $y \stackrel{3}{\imp} x$ is in $\Gamma$, then the maximal version of the {\em left} premise is
 \begin{equation}\label{jujub}
 ((y \imp x) \vee (z \imp x) \vee w) \stackrel{2}{\imp}
 ((y \imp x) \vee w), z \vdash (y \imp x) \vee w.
 \end{equation}
 Applying $\stackrel{2}{\imp}$ right now is impossible: its left premise gets falsified by (\ref{yxzxw}) or (\ref{zyxzxw}). Apply $\vee R_l$ (recall that $\vee R_r$ is used only directly below axiom) and invert $\imp R$:
 $$
 ((y \imp x) \vee (z \imp x) \vee w) \stackrel{2}{\imp}
 ((y \imp x) \vee w), z, y \vdash x.
 $$
 Here the left premise of $\stackrel{2}{\imp}$ is also falsified by
 (\ref{yxzxw}), (\ref{zyxzxw}), or (\ref{zyyxzxw}).

 \vskip 10pt
 
 {\bf Case 2.} 
 Consider again two cases, depending on whether $y \stackrel{3}{\imp} x$ goes to $\Pi$ or to $\Gamma$. If it goes to $\Pi$, then the right premise is, at maximum,
 $$
 (y \imp x) \vee w,
 ((y \imp x) \vee w) \stackrel{4}{\imp} ((z \imp x) \vee w) \vdash
 ((y \vee z) \imp x) \stackrel{1}{\vee} w.
 $$
 Invert $\vee L$ and choose $y \imp x$:
 \begin{equation}\label{alice}
 y \stackrel{5}{\imp} x, ((y \imp x) \vee w) \stackrel{4}{\imp} ((z \imp x) \vee w) \vdash
 ((y \vee z) \imp x) \stackrel{1}{\vee} w.
 \end{equation}
 
 For reusal of our reasoning in further cases, we shall falsify a stronger sequent
 \begin{equation}\label{rabbit}
 y \stackrel{5}{\imp} x, ((y \imp x) \vee (z \imp x) \vee w) \stackrel{4}{\imp} ((z \imp x) \vee w) \vdash
 ((y \vee z) \imp x) \stackrel{1}{\vee} w.
 \end{equation}

 Indeed, $(y \imp x) \vee w \vdash (y \imp x) \vee (z \imp x) \vee w$, and therefore $((y \imp x) \vee (z \imp x) \vee w) \imp ((z \imp x) \vee w) \vdash ((y \imp x) \vee w) \imp ((z \imp x) \vee w)$ is derivable in $\IAL$. Thus,
 if (\ref{alice}) happens to be derivable then, by cut, so will
 be (\ref{rabbit}).

 Now we decompose one of $\stackrel{1}{\vee}$,
 $\stackrel{4}{\imp}$, $\stackrel{5}{\imp}$ in~(\ref{rabbit}).
 
 {\bf Subcase 2--$\Pi$--1.}
 Recall that we never choose $w$ in $\vee R$, and invert $\imp R$:
 $$
 y \stackrel{5}{\imp} x, 
 ((y \imp x) \vee (z \imp x) \vee w) \stackrel{4}{\imp} ((z \imp x) \vee w), 
 y \vee z \vdash x.
 $$
 Invert $\vee L$ and choose $z$:
 $$
 y \stackrel{5}{\imp} x, 
 ((y \imp x) \vee (z \imp x) \vee w) \stackrel{4}{\imp} ((z \imp x) \vee w), 
  z \vdash x.
 $$

 {\em Subsubcase 2--$\Pi$--1--5.}
 The left premise is, at maximum, $$((y \imp x) \vee (z \imp x) \vee w)
 \stackrel{4}{\imp} ((z \imp x) \vee w), z \vdash y.$$
 Applying $\stackrel{4}{\imp} L$ is impossible, since
 its right premise is falsified by choosing $w$:
 $w,z \not\vdash y$ and $w \not\vdash y$.
 
 {\em Subsubcase 2--$\Pi$--1--4.}
 Again, if $y \stackrel{5}{\imp} x$ goes to the new $\Pi$, then the right premise is, at maximum, $z, (z \imp x) \vee w \vdash x,$
 which is falsified by choosing $w$. If it goes to the new $\Gamma$, then the new {\em left} premise is, at maximum, $z \vdash (y \imp x) \vee (z \imp x) \vee w$, which is not derivable by~(\ref{zyxzxw}).
 
 {\bf Subcase 2--$\Pi$--4.}
 If the new $\Pi$ is empty, then the left premise is falsified 
 by~(\ref{yxzxw}). Otherwise, the right premise is
 $$
 (z \imp x) \vee w \vdash ((y \vee z) \imp x) \vee w.
 $$
 Invert $\vee L$ and choose $z \imp x$:
 $$
 z \imp x \vdash ((y \vee z) \imp x) \vee w.
 $$
 Applying $\imp L$ is impossible ($\not\vdash z$); also
 $z \imp x \not\vdash w$. Thus, we have to use $\vee R_l$, and we can immediately apply $\imp R$ afterwards:
 $
 z \imp x, y \vee z \vdash x.
 $
 Inverting $\vee L$ and choosing $y$ falsifies this sequent:
 $z \imp x, y \not\vdash x$.
 
 {\bf Subcase 2--$\Pi$--5.}
 The left premise is, at maximum,
 $$
 ((y \imp x) \vee (z \imp x) \vee w) \imp 
 ((z \imp x) \vee w) \vdash y.
 $$
 This is not derivable.
 
\vskip 5pt
Now let, in Case~2, $y \stackrel{3}{\imp} x$ go to $\Gamma$.
Then the left premise is, at maximum,
$$
((y \imp x) \vee w) \stackrel{4}{\imp} ((z \imp x) \vee w) \vdash
(y \imp x) \vee (z \imp x) \vee w.
$$
This sequent is stronger than~(\ref{three})---that is, (\ref{three}) can be obtained from it by weakening. Therefore, it cannot be derivable, since we've already falsified~(\ref{three}) in Case~1.

\vskip 10pt

{\bf Case 3.} Take the maximal possible $\Pi$ and consider the left premise:
$$
((y \imp x) \vee (z \imp x) \vee w) \stackrel{2}{\imp}
((y \imp x) \vee w), ((y \imp x) \vee w) \stackrel{4}{\imp} 
((z \imp x) \vee w) \vdash y.
$$
This sequent is stronger than~(\ref{vorpal}), and therefore not derivable: (\ref{vorpal}) was falsified in Case~1.

{\bf Case 4.} If $y \stackrel{3}{\imp} x$ goes to $\Pi$, then the maximal version of the right premise of $\stackrel{4}{\imp} L$ is
$$
((y \imp x) \vee (z \imp x) \vee w) \stackrel{2}{\imp} ((y \imp x) \vee w), (z \imp x) \vee w \vdash ((y \vee z) \imp x) 
\stackrel{1}{\vee} w. 
$$
Invert $\vee L$ and choose $z \imp x$:
$$
((y \imp x) \vee (z \imp x) \vee w) \stackrel{2}{\imp} ((y \imp x) \vee w), z \stackrel{6}{\imp} x \vdash ((y \vee z) \imp x) 
\stackrel{1}{\vee} w. 
$$
Suppose this sequent is derivable. Then it will also be derivable after swapping variables $y$ and $z$:
$$
((z \imp x) \vee (y \imp x) \vee w) \imp ((z \imp x) \vee w), y \imp x \vdash ((z \vee y) \imp x) 
\vee w. 
$$
This sequent, however, is exactly~(\ref{rabbit}), up to commutativity; (\ref{rabbit}) was falsified in Case~2.

Finally, if $y \stackrel{3}{\imp} x$, in Case~4, goes to $\Gamma$, then the maximal version of the left premise of $\stackrel{4}{\imp} L$ is 
$$
((y \imp x) \vee (z \imp x) \vee w) \imp
((y \imp x) \vee w) \vdash ((y \vee z) \imp x) \vee w.
$$
This sequent is stronger than~(\ref{jujub}) and therefore cannot be derivable.
\end{proof}

\section{Undecidability with $\BS$, $\wedge$, and $\One$}\label{S:unit}

\subsection{The System $\LLe(\BS,\wedge,\One)$ and Its Undecidability}

In this section we consider the extension of the Lambek calculus with the multiplicative unit constant. The language of our fragment will be as follows:
$\BS$, $\wedge$, $\One$. As shown by Buszkowski~\cite{Buszko1982}, in the fragment of $\BS$ and $\wedge$ the Lambek calculus with empty antecedents
is complete w.r.t.\ L$\varepsilon$-models. As noticed
in the Introduction, however, this is not the case if we add $\One$. In L$\varepsilon$-models, because of the principle $A \cdot \One \vdash \One$, the unit constant $\One$ is 
necessarily interpreted as the singleton set $\{\varepsilon\}$, where $\varepsilon$ is the empty word. (In the presence of the unit constant, we allow the empty word
to belong to our languages and abolish Lambek's non-emptiness restriction.) This particular interpretation of $\One$ satisfies certain principles, including
$A \cdot \{ \varepsilon \} = \{ \varepsilon \} \cdot A$ and $\{ \varepsilon \} \cdot \{ \varepsilon \} = \{ \varepsilon \}$. Moreover, these principles keep valid for 
languages of the form $\{\varepsilon\} \cap B$ (for any $B$). Indeed, this language is either $\{\varepsilon\}$ or $\varnothing$, and for the empty set $\varnothing$ we
also have $A \cdot \varnothing = \varnothing \cdot A$ and $\varnothing \cdot \varnothing = \varnothing$.

Below we present a calculus denoted by $\LLe(\BS,\wedge,\One)$, which reflects these principles as sequential rules:
$$
\infer[Id]{A \vdash A}{}
\qquad
\infer[\One]{A, \One \vdash A}{}
$$
$$
\infer[\BS L]{\Gamma, \Pi, A \BS B, \Delta \vdash C}
{\Pi \vdash A & \Gamma, B, \Delta \vdash C}
\qquad
\infer[\BS R]{\Pi \vdash A \BS B}
{A, \Pi \vdash B}
$$
$$
\infer[\wedge L_l]{\Gamma, A \vee B \vdash C}{\Gamma, A \vdash C}
 \qquad
 \infer[\wedge L_r]{\Gamma, A \vee B \vdash C}{\Gamma, B \vdash C}
 \qquad 
 \infer[\wedge R]{\Pi \vdash A \wedge B}{\Pi \wedge A & \Pi \wedge B}
$$
$$
\infer[L\varepsilon]
{\Gamma, \One \wedge G, A, \Delta \vdash C}
{\Gamma, A, \One \wedge G, \Delta \vdash C}
\quad
\infer[R\varepsilon]
{\Gamma, A, \One \wedge G, \Delta \vdash C}
{\Gamma, \One \wedge G, A, \Delta \vdash C}
\quad
\infer[D\varepsilon]
{\Gamma, \One \wedge G, \Delta \vdash C}
{\Gamma, \One \wedge G, \One \wedge G, \Delta \vdash C}
$$

The rules $L\varepsilon$ and $R\varepsilon$ are called ``commuting'' rules; they reflect the fact that, for any set $X$, 
$X \cdot \{\varepsilon\} = \{\varepsilon\} \cdot X$ and $X \cdot \varnothing = \varnothing \cdot X$. The ``doubling'' rule
$D \varepsilon$ is caused by $\{\varepsilon\} \cdot \{\varepsilon\} = \{\varepsilon\}$ and $\varnothing \cdot \varnothing =
\varnothing$. Thus, these rules express natural algebraic properties of the interpretation of $\One$ as $\varnothing$.
However, we emphasize that they are not admissible in the standard calculus $\LU$, introduced by Lambek~\cite{Lambek69}, that is,
non-commutative intuitionistic multiplicative-additive linear logic.

The rules $L\varepsilon$, $R\varepsilon$, and $D \varepsilon$ are not new. Their underlying principles, namely, $(\One \wedge G) \cdot A \equiv A \cdot (\One \wedge G)$ and $(\One \wedge G) \cdot (\One \wedge G) \equiv \One \wedge G$ appear in works of the Hungarian school (Andr\'{e}ka, Mikul\'{a}s, N\'{e}meti). Namely, in~\cite{Andreka2011TCS} one can find the first of these equivalences (denoted there as formula 3.2), as one of the principles which is true in language algebras, but not in algebras of binary relations. 
The second equivalence is true for binary relations also; formula (CbI) in~\cite{Andreka2011AlgUniv} is actually its stronger version, $(\One \wedge G) \cdot (\One \wedge F) \equiv \One \wedge G \wedge F$. We get our $(\One \wedge G) \cdot (\One \wedge G) \equiv \One \wedge G$ by taking $F = G$.

Andr\'{e}ka, Mikul\'{a}s, and Sain~\cite{Andreka1994LIF} also sketch an undecidability proof for a system related to the one considered here. Their proof is based on the technique of Kurucz et al.~\cite{Kurucz1993}. The system considered in~\cite{Andreka1994LIF} is the logic of residuated distributive lattices over monoids. Unlike the case we consider in this section, their system requires product, the unit and also the zero constant (the minimal element of the lattice) to be present in the language. Here we require only division, additive conjunction, and the unit. The trade-off is that we consider a narrower class of models. Namely, we consider only L$\varepsilon$-models, and these models, as shown above, allow extra principles for $\One$.

We do not claim that $\LLe(\BS,\wedge,\One)$ is an L$\varepsilon$-complete system. Indeed, the L$\varepsilon$-complete extension of $\LU$ happens to
be quite involved (cf.~\cite{KuznJANCL}). In particular, it is still an open problem whether such a complete system is recursively enumerable.
The cut rule is not included in $\LLe(\BS,\wedge,\One)$, so all our derivations will be cut-free. We do not claim that cut is admissible in
this system.

We prove undecidability for the whole range of systems between $\LLe(\BS,\wedge,\One)$ and the L$\varepsilon$-complete system in the language of $\BS$, $\wedge$, $\One$.

\begin{theorem}\label{Th:undec}
 Let $\mathcal{L}$ be an L$\varepsilon$-sound logic which includes $\LLe(\BS,\wedge,\One)$. Then the derivability problem for $\mathcal{L}$ is undecidable.
\end{theorem}

Our undecidability proof is based on encoding computations
of 2-counter Minsky machines~\cite{Kanovich1995Minsky}. In the forward encoding, from Minsky computations to derivations in our calculus, we present explicit derivations in $\LLe(\BS,\wedge,\One)$. For the backwards direction, from derivations to computations, we use a semantic approach using L-models (cf.~\cite{Lafont,LafontScedrov,OkadaTerui}, where phase semantics was used for similar purposes). Thus, we get undecidability not only for $\LLe(\BS,\wedge,\One)$ itself, but for the whole range of its L$\varepsilon$-sound extensions.

Before going further, let us introduce the {\em relative double negation} construction. We fix a variable (atomic proposition) $b$ and define relative negation $A^b$ as
$$
A^b = A \BS b.
$$
The term ``negation'' here is motivated as follows. In linear logic with the falsity constant $\bot$, negation is expressed as $A^\bot = A \imp \bot$. Here we do the same non-commutatively, but due to lack of the $\bot$ constant we replace it by a fixed variable. This is the minimal logic approach: variable $b$ can be read as ``false,'' but no specific axioms like $b \vdash A$ ({\em ex falso}) are imposed for $b$.

The relative double negation now is
$$
A^{bb} = (A \BS b) \BS b.
$$
Notice the difference from the more usual in the Lambek calculus ``type raising'' version of something like double negation: ${}^b A^{b} = b \SL (A \BS b)$. 
In our setting, we have neither $A^{bb} \vdash A$ (due to the intuitionistic nature of the Lambek calculus), nor $A \vdash A^{bb}$ (due to non-commutativity; in contrast, $A \vdash {}^b A^{b}$ is derivable). Nevertheless, $A^{bb}$ will be useful for our construction.

Given a sequence of formulae $\Phi = A_1, A_2, \ldots, A_{m-1}, A_m$ and a formula $C$, we introduce the notation
$$
\Phi \BS C = A_m \BS (A_{m-1} \BS \ldots \BS (A_2 \BS (A_1 \BS C)) \ldots ).
$$
In particular,
$$
\Phi^b = A_m \BS (A_{m-1} \BS \ldots \BS (A_2 \BS (A_1 \BS b)) \ldots )
$$
and 
$$
\Phi^{bb} = \bigl( A_m \BS (A_{m-1} \BS \ldots \BS (A_2 \BS (A_1 \BS b)) \ldots ) \bigr) \BS b.
$$

In what follows, we suppose that the $^{bb}$ operation has a higher priority than ordinary division $\BS$.

Consider a non-deterministic Minsky machine $\mathfrak{M}$ with a finite set of states $\{ L_0, L_1, \ldots, L_n \}$. A configuration of $\mathfrak{M}$ is a triple $(L_i, k_1, k_2)$, where $L_i$ is the current state and $k_1$ and $k_2$ are the current values of $\mathfrak{M}$'s two counters. The counters themselves are denoted by $c_1$ and $c_2$. The configuration $(L_0,0,0)$ is considered the {\em final} one; the initial configuration can be taken arbitrarily. 

Configurations of Minsky machines are encoded as follows. We introduce distinct variables $e_1$, $e_2$, $p_1$, $p_2$, $l_0$, $l_1$, \ldots, $l_n$ and represent configuration $(L_i, k_1, k_2)$ as
$$
e_1, \underbrace{p_1, \ldots, p_1}_{\text{$k_1$ times}}, l_i,
\underbrace{p_2, \ldots, p_2}_{\text{$k_2$ times}}, e_2.
$$
In particular, the final configuration $(L_0,0,0)$ is represented as $e_1, l_0, e_2$.

Minsky instructions are encoded according to the following table:
\begin{center}
 \begin{tabular}{c|c}
 {\em Instruction $I$} & {\em Formula $A_I$} \\ \hline
{\sc inc}$(L_i,1,L_j)$  & 
 $l_i \BS (p_1, l_j)^{bb}$ \\
{\sc inc}$(L_i,2,L_j)$ &  
  $l_i \BS (l_j, p_2)^{bb}$ \\
{\sc dec}$(L_i,1,L_j)$ &  
  $(p_1, l_i) \BS l_j^{bb}$ \\
{\sc dec}$(L_i,2,L_j)$ &  
  $(l_i, p_2) \BS l_j^{bb}$ \\
{\sc jz}$(L_i, 1, L_j)$ &  
  $(e_1, l_i) \BS (e_1, l_j)^{bb}$ \\
{\sc jz}$(L_i, 2, L_j)$ &  
  $(l_i, e_2) \BS (l_j, e_2)^{bb}$
 \end{tabular}
\end{center}

Here instruction {\sc inc}$(L_i,r,L_j)$ {\em (increment)} means ``at state $L_i$, increase $c_r$ by 1 and go to $L_j$'' ($r = 1,2$). Instruction {\sc dec}$(L_i,r,L_j)$ {\em (decrement)} means ``at state $L_i$, decrease $c_r$ by 1 and go to $L_j$.'' If $k_r = 0$, then this instruction cannot be applied. Finally, {\sc jz}$(L_i,r,L_j)$ {\em (zero-test)} means ``at state $L_i$, if $k_r = 0$, go to $L_j$.'' Now if $k_r \ne 0$, then the instruction cannot be applied.

Notice that our version of zero-test and decrement instructions are very restrictive. Once the counter has a wrong value (zero for decreasing or non-zero for zero-test), the machine just fails to proceed. Usually, in such cases the machine is allowed to perform conditional branching (e.g., zero-test jumps to $L_j$ if the counter is zero and safely stays at $L_i$ if not). These restrictions, however, are compensated by the allowed non-determinism of $\mathfrak{M}$. Indeed, the compound {\sc jzdec}$(L_i, r, L_{j_1}, L_{j_2})$ instruction from Minsky's original formalism~\cite{Minsky1961}, meaning ``at state $L_i$, if $k_r \ne 0$, decrease $c_r$ by one and go to $L_{j_1}$, and if $k_r = 0$, go to $L_{j_2}$,'' is modelled by adding simultaneously two instructions: 
{\sc dec}$(L_i, r, L_{j_1})$ and {\sc jz}$(L_i, r, L_{j_2})$. This non-deterministically branches computation; however, exactly one branch (depending on whether $k_r = 0$) could be successful, the other one immediately fails.

Let us denote the set of our variables, except $b$,  by $\Var$:
$$
\Var = \{ e_1, e_2, p_1, p_2, l_0, l_1, \ldots, l_n \}.
$$

Finally, the Minsky machine $\mathfrak{M}$ is represented by the following formula
$$
G = ((e_1, l_0, e_2) \BS b) \wedge \bigwedge_{I} A_I \wedge \bigwedge_{q \in \Var} (q \BS q^{bb}).
$$
Here in the first big conjunction $I$ ranges among all instructions of $\mathfrak{M}$.

Now we are ready to state our main encoding theorems.

\begin{theorem}\label{Th:encforw}
If $\mathfrak{M}$ can reach the final configuration $(L_0,0,0)$, starting from $(L_i,k_1,k_2)$, then the following sequent is derivable in $\LLe(\BS,\wedge,\One)$:
$$
\One \wedge G, e_1, \underbrace{p_1, \ldots, p_1}_{\text{$k_1$ times}}, l_i,
\underbrace{p_2, \ldots, p_2}_{\text{$k_2$ times}}, e_2 \vdash b.\eqno{(*)}
$$
\end{theorem}

\begin{theorem}\label{Th:encback}
 If the sequent $(*)$ is true in all L$\varepsilon$-models, then $\mathfrak{M}$ can reach $(L_0,0,0)$ from $(L_i,k_1,k_2)$.
\end{theorem}

Notice that our encodings are in a sense ``upside-down'': the starting configuration corresponds to the goal sequent in our derivation, and the sequent encoding the final configuration $(L_0,0,0)$ is on the top of the derivation, very close to axioms (see proof of Theorem~\ref{Th:encforw} below). The right intuition here is to consider the derivation in the direction of proof search, developing from the goal up to axioms.
This direction correctly reflects the direction of Minsky computation.

Theorem~\ref{Th:undec} (our undecidability result) immediately follows from Theorems~\ref{Th:encforw} and~\ref{Th:encback}. Indeed, if $\Lc$ is a logic which is L$\varepsilon$-sound and includes $\LLe(\BS,\wedge,\One)$, then $(*)$ is provable in $\Lc$ if and only if $\Mf$ can reach $(L_0,0,0)$ from $(L_1,k_1,k_2)$. Indeed, the ``if'' direction is by Theorem~\ref{Th:encforw}, and 
the ``only if'' direciton is by Theorem~\ref{Th:encback}. Since reachability in Minsky computations is undecidable, we get undecidability of $\Lc$.

Before proving Theorems~\ref{Th:encforw} and~\ref{Th:encback}, we establish several technical results.

Notice that each formula in the big conjunction $G$, except the first one, is of the form $G_{\Phi,\Psi} = \Psi \BS \Phi^{bb}$. 
The key lemma for such formulae, in the view of Theorem~\ref{Th:encforw}, is as follows.

\begin{lemma}\label{Lm:G}
 If the big conjunction $G$ includes $G_{\Phi,\Psi}$ and 
 $\One \wedge G, \Phi, \Delta \vdash b$ is derivable in $\LLe(\BS,\wedge,\One)$, then so is $\One \wedge G, \Delta, \Psi \vdash b$.
\end{lemma}

\begin{proof}
 The derivation is as follows:
 $$
 \infer[D\varepsilon]{\One \wedge G, \Delta, \Psi \vdash b}
 {\infer[\wedge L\text{ \small several times}]
 {\One \wedge G, \One \wedge G, \Delta, \Psi, \vdash b}
 {\infer[L\varepsilon\text{ \small several times}]
 {\One \wedge G, \One \wedge (\Psi \BS \Phi^{bb}), \Delta, \Psi \vdash b}
 {\infer[\wedge L_r]{\One \wedge G, \Delta, \Psi, \One \wedge (\Psi \BS \Phi^{bb}) \vdash b}
 {\infer[\BS L]{\One \wedge G, \Delta, \Psi, \Psi \BS \Phi^{bb} \vdash b}
 {\Psi \vdash \Psi & 
 \infer[\BS L]{\One \wedge G, \Delta, (\Phi \BS b) \BS b \vdash B}
 {\infer[\BS R]{\One \wedge G, \Delta \vdash \Phi \BS b}
 {\infer[R \varepsilon\text{ \small several times}]{\Phi, \One \wedge G, \Delta \vdash b}
 {\One \wedge G, \Phi, \Delta \vdash b} & b \vdash b}}}}}}}
 $$
\end{proof}

\begin{corollary}[``Post-ish productions'']\label{Cor:Postish}
 Let $\Delta_1$ and $\Delta_2$ be sequences of variables from $\Var$ (no complex formulae). Then, provided that $G$ includes $q \BS q^{bb}$ for any $q \in \Var$, the sequent $\One \wedge G, \Delta_2, \Delta_1 \vdash b$ is derivable in $\LLe(\BS,\wedge,\One)$ from
 $\One \wedge G, \Delta_1, \Delta_2 \vdash b$.
\end{corollary}

\begin{proof}
 It is sufficient to consider the case of $\Delta_1 = q$; then we proceed by induction on the length of $\Delta_1$. For $\Delta_1 = q$, we apply Lemma~\ref{Lm:G} with $\Phi = \Psi = q$.
\end{proof}

\begin{corollary}[One step of Minsky computation]\label{Cor:step}
 Suppose the Minsky machine $\mathfrak{M}$ can make a computation step from configuration $(L_i, k_1, k_2)$ to 
 configuration $(L_{i'}, k'_1, k'_2)$, and let $(*')$ be the instance of $(*)$ for $(L_{i'}, k'_1, k'_2)$. Then $(*)$ is derivable from $(*')$ in $\LLe(\BS,\wedge,\One)$.
\end{corollary}

\begin{proof}
 The proof is performed uniformly for all Minsky instructions. For any instruction $I$, the corresponding formula $A_I$ is of the form $G_{\Phi,\Psi} = \Psi \BS \Phi^{bb}$. On the other hand, $(*')$ is obtained from $(*)$ by replacing $\Psi$ with $\Phi$ in the antecedent. 
 
 For example, for the instruction {\sc inc}$(L_i,1,L_j)$ in the center of $(*)$ we have $l_i = \Psi$, which is replaced with $p_1, l_j = \Phi$ in $(*')$. This exactly corresponds to the computation step: the number of $p_1$'s (that is, the value of $c_1$) gets increased by 1, and the state is changed to $l_j$. 
 For {\sc jz}, the replacement happens at the edge of the antecedent, involving $e_1$ or $e_2$.
 
 Thus, $(*)$ is of the form $\One \wedge G, \Delta_1, \Psi, \Delta_2 \vdash b$ and $(*')$ is $\One \wedge G, \Delta_1, \Phi, \Delta_2 \vdash b$. Now we derive $(*)$ from $(*')$ in the following way:
 $$
 \infer
 [\text{\small Corollary~\ref{Cor:Postish}}]
 {\One \wedge G, \Delta_1, \Psi, \Delta_2 \vdash b}
 {\infer
 [\text{\small Lemma~\ref{Lm:G}}]
 {\One \wedge G, \Delta_2, \Delta_1, \Psi \vdash b}
 {\infer[\text{\small Corollary~\ref{Cor:Postish}}]
 {\One \wedge G, \Phi, \Delta_2, \Delta_1 \vdash b}
 {\One \wedge G, \Delta_1, \Phi, \Delta_2 \vdash b}}}
 $$
\end{proof}

Now we are ready to prove Theorem~\ref{Th:encforw}.

\begin{proof}[Proof of Theorem~\ref{Th:encforw}]
 Using Corollary~\ref{Cor:step} and induction on the number of steps in Minsky computation from $(L_i, k_1, k_2)$ to $(L_0,0,0)$, we derive $(*)$ from 
 $$
 \One \wedge G, e_1, l_0, e_2 \vdash b
 \eqno{(*_0)}
 $$
 This sequent $(*_0)$ is derived as follows:
 $$
 \infer[L\varepsilon\text{ \small 3 times}]
 {\One \wedge G, e_1, l_0, e_2 \vdash b}
 {\infer
 [\wedge L\text{ \small several times}]
 {e_1, l_0, e_2, \One \wedge G \vdash b}
 {\infer[\BS L]{e_1, l_0, e_2, e_2 \BS (l_0 \BS (e_1 \BS b)) \vdash b}
 {e_2 \vdash e_2 & 
 \infer[\BS L]{e_1, l_0, l_0 \BS (e_1 \BS b) \vdash b}
 {l_0 \vdash l_0 &
 \infer[\BS L]{e_1, e_1 \BS b \vdash b}
 {e_1 \vdash e_1 & b \vdash b}}}}}
 $$
\end{proof}

The backwards direction, Theorem~\ref{Th:encback}, is proved by constructing a specific L$\varepsilon$-model. 
Let $\Sigma = \Var$ and fefine $B_{\mathfrak{M}}$ as the set of ``terminating words'' for $\mathfrak{M}$, defined as follows:
$$
B_{\mathfrak{M}} = 
\{ e_1 \underbrace{p_1 \ldots p_1}_{\text{$k_1$ times}} l_i
\underbrace{p_2 \ldots p_2}_{\text{$k_2$ times}} e_2 \mid
\mbox{ $\mathfrak{M}$ can reach $(L_0,0,0)$ from $(L_i,k_1,k_2)$ }\}. 
$$
Now  define the interpreting function $w$ on variables as follows:
\begin{align*}
 & w(q) = \{ q \} \qquad\mbox{for $q \in \Var$;}\\
 & w(b) = \{ \Xi\Upsilon \mid \mbox{ $\Xi$ and $\Upsilon$ are words over
 $\Sigma$ such that $\Upsilon\Xi \in B_{\mathfrak{M}}$ } \}.
\end{align*}

\begin{lemma}\label{Lm:Geps}
 $w(\One \wedge G) = \{ \varepsilon \}$.
\end{lemma}

\begin{proof}
 It is sufficient to show that $\varepsilon \in w(G)$, that is, $\varepsilon$ belongs to interpretation of all formulae in the big conjunction $G$.
 
 First, $\varepsilon \in w((e_1, l_0, e_2) \BS b)$. Indeed,
 $w((e_1,l_0,e_2) \BS b) = 
 \{ e_1 l_0 e_2 \} \BS w(b)$, thus, 
 we have to show that $e_1 l_0 e_2 \varepsilon = e_1 l_0 e_2 \in w(b)$. This is indeed so by the definition of $B_{\mathfrak{M}}$, since $(L_0,0,0)$ is reachable from itself in zero steps.
 
 Second, we prove that $\varepsilon \in w(A_I)$ for each instruction $I$ of $\mathfrak{M}$. Recall that $A_I = \Psi \BS \Phi^{bb} = \Psi \BS ((\Phi \BS b) \BS b)$, and if instruction $I$ changes the configuration from $(L_i, k_1, k_2)$ to $(L_{i'}, k'_1, k'_2)$, then the code of the second configuration is obtained from the code of the first one by replacing $\Psi$ with $\Phi$. 
 In other words, the code of $(L_i, k_1, k_2)$ is $\Delta_1 \Psi \Delta_2$ and the code of $(L_{i'}, k'_1, k'_2)$ is $\Delta_1 \Phi \Delta_2$. We have to prove that $\varepsilon \in w(\Psi) \BS w((\Phi \BS b) \BS b)$.
 Since $w(\Psi) = \{ \Psi \}$ ($\Psi$ contains only letters from $\Var$), this means that $\Psi$ should belong to 
 $w((\Phi \BS b) \BS b)$. 
 
 In turn, $\Psi \in w((\Phi \BS b) \BS b)$ means 
 that for any word $\Delta \in w(\Phi \BS b)$ we have
 $\Delta \Psi \in w(b)$. The fact that $\Delta \in w(\Phi \BS b)$, since $w(\Phi) = \{ \Phi \}$, actually means that $\Phi\Delta \in w(b)$. Thus, we have to prove, for an arbitrary $\Delta$, that if $\Phi\Delta \in w(b)$, then $\Delta\Psi \in w(b)$.
 
 If $\Phi\Delta \in w(b)$, then we have $\Delta = \Delta_1\Delta_2$, and $\Delta_2 \Phi \Delta_1 \in B_{\mathfrak{M}}$. Here $\Phi$ cannot be split between $\Xi$ and $\Upsilon$, because any word in $B_{\mathfrak{M}}$ should begin with $e_1$ and end on $e_2$. This means that $\Delta_2 \Phi \Delta_1$ is a code of some configuration $(L_{i'}, k'_1, k'_2)$, from which $\mathfrak{M}$ can reach the final configuration. As noticed above, this means that $\Delta_2 \Psi \Delta_1$ encodes a configuration $(L_i, k_1, k_2)$, which transforms into $(L_{i'}, k'_1, k'_2)$ by applying instruction~$I$. Therefore, from $(L_i, k_1, k_2)$ our Minsky machine can also reach the final state, hence $\Delta_2 \Psi \Delta_1 \in B_{\mathfrak{M}}$. This yields $\Delta\Psi = \Delta_1 \Delta_2 \Psi \in w(b)$, which is our goal.
 
 Third, consider $q \BS q^{bb}$, where $q \in \Var$. We have to show that $\varepsilon \in w(q) \BS w(q^{bb})$, that is, $q \in w(q^{bb})$. The latter means that for any $\Delta \in w(q \BS b)$ the word $\Delta q$ should belong to $w(b)$. 
 This is indeed so: if $\Delta \in w(q \BS b)$, then $q \Delta \in w(b)$, and since $w(b)$ is closed under cyclic transpositions, also $\Delta q \in w(b)$.
\end{proof}

Now we are ready to prove Theorem~\ref{Th:encback}.

\begin{proof}[Proof of Theorem~\ref{Th:encback}]
 If $(*)$ is true in all L$\varepsilon$-models, it is true in the specific model defined above. By Lemma~\ref{Lm:Geps}, $w(\One \wedge G) = \{ \varepsilon \}$;
 $w(q) = \{ q \}$ for any $q \in \Var$. 
Thus, we have
 $$
 e_1 \underbrace{p_1 \ldots p_1}_{\text{$k_1$ times}} l_i 
 \underbrace{p_2 \ldots p_2}_{\text{$k_2$ times}} e_2 \in w(b),
 $$
 and therefore
 $$
 e_1 \underbrace{p_1 \ldots p_1}_{\text{$k_1$ times}} l_i 
 \underbrace{p_2 \ldots p_2}_{\text{$k_2$ times}} e_2 \in B_{\mathfrak{M}}.
 $$
 (No cyclic transposition is possible, since $e_1$ and $e_2$ should start and end the word.)
 
 By definition of $B_{\mathfrak{M}}$, this means that $\mathfrak{M}$ can reach the final state $(L_0,0,0)$, starting from $(L_i,k_1,k_2)$.
\end{proof}

\subsection{Models on Regular Languages with the Unit Constant}

Let $\Th(\mbox{L$\varepsilon$-models}; \BS, \wedge, \One)$ denote the set of all sequents in the language of $\BS$, $\wedge$, $\One$ which are true in all L$\varepsilon$-models, that is, the {\em complete theory} of this class of models.

As noticed above, the question of axiomatizing this theory is 
quite involved. We know that this theory includes $\LLe(\BS,\wedge,\One)$, introduced in the previous section, but it is  probably much more complicated. For example, as shown in~\cite{KuznJANCL}, Soboci\'{n}ski's 3-valued logic $\mathbf{RM}_3$ can be embedded into $\Th(\mbox{L$\varepsilon$-models}; \BS, \wedge, \One)$. 

It follows from Theorem~\ref{Th:undec} that $\Th(\mbox{L$\varepsilon$-models}; \BS, \wedge, \One)$ is undecidable. More precisely, it is $\Sigma_1^0$-hard (hard w.r.t.\ the class of recursively enumerable sets). The upper complexity bound, however, is not known: this theory could possibly be even not recursively enumerable.
Having the algorithmic complexity question for $\Th(\mbox{L$\varepsilon$-models}; \BS, \wedge, \One)$ open, we can still obtain an interesting corollary of our complexity estimations.

Recall the standard notion of regular expression. Regular expressions are constructed from constants $0$ and $1$ using two binary operations, $\cdot$ and $+$, and one unary operation, ${}^*$. The language $\mathscr{L}(R)$ described by a given regular expression $R$ is defined recursively:
\begin{align*}
 & \mathscr{L}(0) = \varnothing; \\
 & \mathscr{L}(1) = \{ \varepsilon \}; \\
 & \mathscr{L}(A \cdot B) = \mathscr{L}(A) \cdot \mathscr{L}(B);\\
 & \mathscr{L}(A + B) = \mathscr{L}(A) \cup \mathscr{L}(B); \\
 & \mathscr{L}(A^*) = \bigl(\mathscr{L}(A)\bigr)^* = 
 \{ u_1 \ldots u_n \mid n \ge 0, u_i \in \mathscr{L}(A) \}.
\end{align*}
Languages described by regular expressions are called regular languages.

By {\sc Lreg}$\varepsilon$-models let us denote a subclass of L$\varepsilon$-models in which every variable as interpreted as a {\em regular} language, that is, a set of words described by a regular expression. It is well-known that the class of regular languages is closed under intersection (see, for example, \cite[Theorem~2.8]{AhoUllman}). Moreover, it is also closed under division:
\begin{proposition}
 If $A$ and $B$ are regular languages, then so are $A \BS B$ and $B \SL A$.
\end{proposition}

\begin{proof}
 A more well-known fact (see, for example,~\cite[Exercise 2.3.17a]{AhoUllman}) is that the class of regular languages is closed under the following modified division operation with the existential quantifier instead of the universal one: $A \cBS B = 
\{ u \in \Sigma^* \mid (\exists v \in A) \, vu \in B\}$. Our ``normal'' division $\BS$ can be reduced to $\cBS$ by the complement operation:
$A \BS B = \overline{A \cBS \overline B}$, where
$\overline{X} = \Sigma^* - X$. Since the class of regular languages is closed under $\cBS$ and complement (again, see~\cite[Theorem~2.8]{AhoUllman}), it is also closed under $\BS$.
The $\SL$ case is symmetric.
\end{proof}

Thus, in {\sc Lreg}$\varepsilon$-models interpretations of all formulae are regular languages.

In the language without the unit constant, namely, $\BS$, $\SL$, $\wedge$, the theory of {\sc Lreg}$\varepsilon$-models coincides with the theory of all L$\varepsilon$-models:

\begin{proposition}
$
\Th(\mbox{{\sc Lreg}$\varepsilon$-models}; \BS, \SL, \wedge) =
\Th(\mbox{L$\varepsilon$-models}; \BS, \SL, \wedge).
$
\end{proposition}

\begin{proof}
 On the one hand, the calculus $\MALCs(\BS,\SL,\wedge)$ is sound w.r.t.\ all L$\varepsilon$-models. On the other hand, as shown by Buszkowski~\cite{Buszkowski1996}, it is complete w.r.t.\ a class of models which is even narrower than the class of {\sc Lreg}$\varepsilon$-models. Namely, $\MALCs(\BS,\SL,\wedge)$ is complete w.r.t.\ the class of L$\varepsilon$-models in which variables are interpreted by cofinite languages. (A cofinite language is a language which includes all words over a given alphabet, except for a finite set.) In this case, formulae are interpreted by cofinite or finite languages, and any finite or cofinite language is regular. 
Therefore, both $\Th(\mbox{{\sc Lreg}$\varepsilon$-models}; \BS, \SL, \wedge)$ and $\Th(\mbox{L$\varepsilon$-models}; \BS, \SL, \wedge)$ are axiomatized by the same calculus $\MALCs(\BS,\SL,\wedge)$.
\end{proof}

The unit changes things dramatically. With the unit, there is no completeness result, like Theorem~\ref{Th:compl}, but also no equivalence between theories of all L$\varepsilon$-models and {\sc Lreg}$\varepsilon$-models.

\begin{theorem}\label{Th:regdiff}
 $\Th(\mbox{{\sc Lreg}$\varepsilon$-models}; \BS, \wedge, \One) \ne
\Th(\mbox{L$\varepsilon$-models}; \BS, \wedge, \One).$
\end{theorem}

\begin{proof}
 As follows from Theorem~\ref{Th:undec},  $\Th(\mbox{L$\varepsilon$-models}; \BS, \wedge, \One)$ is $\Sigma_1^0$-hard. On the other hand, following Buszkowski~\cite{Buszko2006RelMiCS}, we can show that $\Th(\mbox{L$\varepsilon$-models}; \BS, \wedge, \One)$ belongs to the $\Pi_1^0$ class. Indeed, a sequent belongs to this theory if and only if it is true in all {\sc Lreg}$\varepsilon$-models. A concrete sequent includes only a finite number of variables, $p_1$, \ldots, $p_n$. Thus, a model for this sequent is defined by a finite number of regular expressions $R_1$, \ldots, $R_n$, which describe the languages $w(p_1)$, \ldots, $w(p_n)$. This means that the general truth condition for this sequent can be written down as the following formula:
 \begin{multline*}
 \forall R_1 \, \ldots \, \forall R_n \bigl(
 \mbox{the sequent is true under interpretation} \\ \mbox{where $w(p_i)$ is the language of $R_i$} \bigr).
 \end{multline*}
 Quantifiers $\forall R_1$, \ldots, $\forall R_n$ can be encoded as quantifiers over natural numbers representing the regular expressions.
 The quantifier-free part of the formula (truth condition under a concrete $w$) is decidable, because all necessary operations on regular expressions are computable. 
 Thus, we get a $\Pi_1^0$ representation of the set of sequents which are true in all {\sc Lreg}$\varepsilon$-models.
 
 It is well known that a set cannot belong to $\Pi_1^0$ and be $\Sigma_1^0$-hard at the same time. (Otherwise, for any set in $\Sigma_1^0$ there would be a computable reduction to a set in $\Pi_1^0$, which would yield $\Sigma_1^0 \subseteq \Pi_1^0$, which is not the case.) Therefore,
 $$\Th(\mbox{{\sc Lreg}$\varepsilon$-models}; \BS, \wedge, \One) \ne
\Th(\mbox{L$\varepsilon$-models}; \BS, \wedge, \One).$$
\end{proof}

Notice that our proof of Theorem~\ref{Th:undec} does not apply to 
$\Th(\mbox{{\sc Lreg}$\varepsilon$-models}; \linebreak \BS, \wedge, \One)$, because the language $w(b)$ there is non-regular (in fact, it is undecidable). 

Since the class of {\sc Lreg}$\varepsilon$-models is narrower than the class of all L$\varepsilon$-models, we have (by Galois connection) an inverted inclusion of theories:
$$\Th(\mbox{{\sc Lreg}$\varepsilon$-models}; \BS, \wedge, \One) \supset
\Th(\mbox{L$\varepsilon$-models}; \BS, \wedge, \One).
$$
By our Theorem~\ref{Th:regdiff}, this inclusion is strict.
Thus, the other inclusion should fail:
$$\Th(\mbox{{\sc Lreg}$\varepsilon$-models}; \BS, \wedge, \One) \not\subset
\Th(\mbox{L$\varepsilon$-models}; \BS, \wedge, \One).
$$
In other words, there exists a sequent which is true in all {\sc Lreg}$\varepsilon$-models, but not in all L$\varepsilon$-models. Our proof, however, is non-constructive, and we do not present a concrete example of such sequent. 
Constructing such a concrete example is left for further research.

Notice that if we apply the reasoning establishing the upper $\Pi_1^0$ bound of $\Th(\mbox{{\sc Lreg}$\varepsilon$-models}; \BS, \wedge, \One)$ to $\Th(\mbox{L$\varepsilon$-models}; \BS, \wedge, \One)$, we shall have to quantify over {\em arbitrary} formal languages $w(p_1)$, \ldots, $w(p_n)$. This results in  hyperarithmetical quantifiers, and yields only a very high, $\Pi_1^1$ complexity upper bound  for $\Th(\mbox{L$\varepsilon$-models}; \BS, \wedge, \One)$.




\section{Concluding Remarks}

In this article, we have investigated language interpretations of natural extensions of the Lambek calculus: with additive operations ($\vee$ and $\wedge$) and with additive conjunction ($\wedge$) and the unit constant ($\One$). 

For extensions with additive connectives (Section~\ref{S:distr}), we have shown that conjunction and disjunction show significantly different behaviour. It is known that adding both conjunction and disjunction leads to incompleteness due to the distributivity law $\mathcal{D}$. This law  is true in all language models, but not derivable in the multiplicative-additive Lambek calculus ($\MALC$). Adding only conjunction, however, still provides completeness. Any sequent in the language of $\BS$, $\SL$, $\wedge$ (but not $\vee$) that is derivable with the help of $\mathcal{D}$, is also derivable without it. For disjunction the situation is opposite: there exists a sequent in the language of $\BS$, $\SL$, $\vee$, which is derivable using $\mathcal{D}$, but not derivable without it.

Moreover, this difference between $\wedge$ and $\vee$ keeps valid for systems with permutation and/or weakening structural rules, that is, intuitionistic linear ($\ILL$), and affine ($\IAL$) logics and affine $\MALC$. 

For the extension of the Lambek calculus with the unit, $\One$, it is well-known that its standard axiomatization in the style of linear logic does not give an L$\varepsilon$-complete system. In Section~\ref{S:unit}, we present a system in the language $\BS, \wedge, \One$, where rules for $\One$ reflect natural algebraic properties of the empty word in the algebra of formal languages.
This system is denoted by $\LLe(\BS,\wedge,\One)$. We do not claim L$\varepsilon$-completeness of $\LLe(\BS,\wedge,\One)$. Instead, we consider the whole range of logics between $\LLe(\BS,\wedge,\One)$ and the L$\varepsilon$-complete system denoted by $\Th(\mbox{L$\varepsilon$-models}; \BS, \wedge, \One)$. For any logic within this range, we show that it is undecidable; more precisely, $\Sigma_1^0$-complete. As a corollary, we also show that, in the language of $\BS, \wedge, \One$, the complete theory of all L$\varepsilon$-models is different from that of {\sc Lreg}$\varepsilon$-models, where formulae are interpreted by regular languages.

A preliminary version of this article was presented at WoLLIC 2019 and published in its lecture notes~\cite{KanKuzSceWoLLICLMod}. 
Let us briefly list the results which are new in this article, compared to the WoLLIC paper.
%
\begin{itemize}
 \item In the language  without additive conjunction, we show incompleteness not only for $\MALC$, but also for its extensions: $\MALCs$, $\AMALC$, $\ILL$, and $\IAL$. 
 \item We prove that $\MALC(\BS,\SL,\wedge)$ is a conservative fragment of $\MALC$ extended with the distributity law $\mathcal{D}$. Moreover, we prove similar results for $\MALCs$, $\AMALC$, $\ILL$, and $\IAL$.
 %
 \item 
 We prove that, in the language including $\One$, the theory of all L$\varepsilon$-models is different from the theory of L{\sc Reg}$\varepsilon$-models, in which formulae are interpreted by regular languages. In the language of $\BS, \SL, \wedge$ (without $\One$), the corresponding theories  coincide due to a completeness result by Buszkowski~\cite{Buszkowski1996}.
\end{itemize}

While in Section~\ref{S:distr} we have presented a quite completed study, Section~\ref{S:unit} leaves many questions open for further investigations. Among these, we would like to emphasize the following ones.
\begin{itemize}
 \item The question of axiomatization, or even recursive enumerability for complete theories $\Th(\mbox{L$\varepsilon$-models}; \BS, \wedge, \One)$ and $\Th(\mbox{{\sc Lreg}$\varepsilon$-models}; \BS, \wedge, \One)$ is still open, and potentially very hard. Notice that these theories are different (Theorem~\ref{Th:regdiff}) and that for $\Th(\mbox{{\sc Lreg}$\varepsilon$-models}; \BS, \wedge, \One)$ enumerability will immediately mean decidability.
 
 \item A possibly easier question would be to construct a concrete formula distinguishing $\Th(\mbox{L$\varepsilon$-models}; \BS, \wedge, \One)$ and $\Th(\mbox{{\sc Lreg}$\varepsilon$-models}; \BS, \wedge, \One)$. That is,  we are looking for an explicit example for Theorem~\ref{Th:regdiff}.
 
 \item Without the unit, we know that $$\MALCs(\BS,\SL,\wedge)= \Th(\mbox{L$\varepsilon$-models}; \BS, \SL, \wedge) = \Th(\mbox{{\sc Lreg}$\varepsilon$-models}; \BS, \SL, \wedge).$$ By the completeness theorem of Pentus~\cite{PentusAPAL,PentusFmonov}, the first equality is also true for the language of $\BS,\SL,\cdot$ (with product instead of additive conjunction). There are two open questions. First, whether Pentus' theorem is true for the language with both conjunctions ($\BS,\SL,\cdot,\wedge$). Second, whether Pentus' theorem is true for  {\sc Lreg}$\varepsilon$-models. Both questions are questions of making Pentus' result stronger. Recalling that Pentus' proofs are quite sophisticated, these questions are also probably very hard.
\end{itemize}

\bibliography{subexponential.bib}

 \end{document}